\documentclass[11pt, letterpaper]{amsart}

\usepackage{amsmath}
\usepackage{amssymb}
\usepackage{amsfonts}
\usepackage{ dsfont }
\usepackage{cases}
\usepackage{cite}
\usepackage[hyphens]{url}
\usepackage{hyperref}
\usepackage{amsthm}
\usepackage{upgreek}
\usepackage{color}

\newtheorem{theorem}{Theorem}[section]

\newtheorem{lemma}[theorem]{Lemma}

\theoremstyle{definition}

\theoremstyle{remark}
\newtheorem{remark}{Remark}[section]

\newcommand{\leref}{Lemma~\ref}

\newcommand{\thref}{Theorem~\ref}

\title[]{On utility maximization with derivatives under model uncertainty}
\author[]{Erhan Bayraktar} \thanks{This research was supported in part by the National Science Foundation under grants DMS 0906257 and DMS 1118673.}  
\address{Department of Mathematics, University of Michigan}
\email{erhan@umich.edu}
\author[]{Zhou Zhou}
\address{Department of Mathematics, University of Michigan}
\email{zhouzhou@umich.edu}
\date{\today}
\begin{document}
\maketitle

\begin{abstract}

We consider the robust utility maximization using a static holding in derivatives and a dynamic holding in the stock. There is no fixed model for the price of the stock but we consider a set of probability measures (models) which are not necessarily dominated by a fixed probability measure.
By assuming that the set of physical probability measures is convex and weakly compact, we obtain the duality result and the existence of an optimizer.
\end{abstract}

\section{Set-up}
We assume that in the market, there is a single risky asset  at discrete times $t=1,\dotso,T$. Let $S=(S_t)_{t=1}^T$ be the canonical process on the path space  $\mathbb{R}_+^T$, i.e, for  $(s_1,\dotso,s_T) \in \mathbb{R}_+^T$ we have that $S_i(s_1,\dotso,s_T)=s_i$. The random variable $S_i$ represents the price of the risky asset  at time $t=i$. We denote the current spot price of the asset as $S_0=s_0$. In addition, we assume that in the market there are a finite number of options $g_i:\ \mathbb{R}_+^T\rightarrow\mathbb{R},\ i=1,\dotso,N$, which can be bought or sold at time $t=0$ at price $g_i^0$. We assume $g_i$ is continuous and $g_i^0=0$. 

Let 
$$\mathcal{M}:=\{\mathbb{Q} \text { probability measure on } \mathbb{R}_+^T:\ S=(S_i)_{i=1}^T \text{ is a } \mathbb{Q}-\text{martingale};$$
$$\text{ for } i=1,\dotso,N,\ \mathbb{E}_\mathbb{Q}g_i=0.\}$$
We make the standing assumption that $\mathcal{M}\neq\emptyset$.

Let us consider the semi-static trading strategies consisting of the sum of a static option portfolio and a dynamic strategy in the stock. We will denote by $\Delta$ the predictable process corresponding to the holdings on the stock. More precisely, the semi-static strategies generate payoffs of the form:
$$x+\sum_{i=1}^N h_i g_i(s_1,\dotso,s_n)+\sum_{j=1}^{T-1}\Delta_j(s_1,\dotso,s_j)(s_{j+1}-s_j)=:x+h\cdot g+(\Delta\cdot S)_T,\ s_1,\dotso,s_T\in\mathbb{R}_+,$$
where $x$ is the initial wealth, $h=(h_1,\dotso,h_N)$ and $\Delta=(\Delta_1,\dotso,\Delta_{T-1})$.

We will assume that $U$ is a function defined on $\mathbb{R}_+$ that is bounded, strictly increasing, strictly concave, continuously differentiable and satisfies the Inada conditions
$$U'(0)=\lim_{x\rightarrow 0}U'(x)=\infty,$$
$$U'(\infty)=\lim_{x\rightarrow\infty}U'(x)=0.$$ 
We also assume that $U$ has asymptotic elasticity strictly less than 1, i.e,
$$AE(U)=\limsup_{x\rightarrow\infty}\frac{xU'(x)}{U(x)}<1.$$

Let $\mathcal{P}$ be a set of probability measures on $\mathbb{R}_+^T$, which represents the possible beliefs for the market. We make the following assumptions on $\mathcal{P}$:\\
\textbf{Assumption P:}
\begin{itemize}
\item[(1)] $\mathcal{P}$ is convex and weakly compact.
\item[(2)] For any $\mathbb{P}\in\mathcal{P}$, there exists a $\mathbb{Q}\in\mathcal{M}$ that is equivalent to $\mathbb{P}$.
\end{itemize} 
Note that the second condition is natural in the sense that every belief  
in the market model is reasonable concerning no arbitrage, e.g., see \cite{Nutz2}. We consider the robust utility maximization problem
$$\hat u(x)=\sup_{(\Delta, h)}\inf_{\mathbb{P}\in\mathcal{P}}\mathbb{E}_\mathbb{P}U\left(x+(\Delta\cdot S)_T+h\cdot g\right).$$

\section{Main result}
\thref{theorem1} and \thref{theorem2} are the main results of this paper. We will first introduce some spaces and value functions concerning the duality.

Let
\begin{equation}
V(y)=\sup_{x>0}[U(x)-xy)],\ \ \ y>0,
\notag\end{equation}
and
\begin{equation}
I:=-V'=(U')^{-1}.
\notag\end{equation}
For any $\mathbb{P}\in\mathcal{P}$, we define some spaces as follows, where the (in)equalities are in the sense of $\mathbb{P}$-a.s.
\begin{itemize}
\item $\mathfrak{X}_\mathbb{P}(x,h)=\{X:\ X_0=x,\ x+(\Delta\cdot S)_T+h\cdot g\geq 0,\ \text{for some }\Delta\} $
\item $\mathcal{Y}_\mathbb{P}(y)=\{Y\geq 0:\ Y_0=y,\ XY \text{ is a $\mathbb{P}$-super-martingale, }\forall X\in\mathfrak{X}_\mathbb{P}(1,0)\}$
\item $\mathfrak{Y}_\mathbb{P}(y)=\{Y\in\mathcal{Y}_\mathbb{P}(y):\ \mathbb{E}_\mathbb{P}\left(Y_T(X_T+h\cdot g)
\right)\leq xy, \forall X\in\mathfrak{X}_\mathbb{P}(x,h)\}$
\item $\mathcal{C}_\mathbb{P}(x,h)=\{c\in L_+^0(\mathbb{P}):\ c\leq X_T+h\cdot g, \text{ for some } X\in\mathfrak{X}_\mathbb{P}(x,h)\}$
\item $\mathcal{C}_\mathbb{P}(x)=\bigcup_h\mathcal{C}_\mathbb{P}(x,h)$
\item $\mathcal{D}_\mathbb{P}(y)=\{d\in L_+^0(\mathbb{P}):\ d\leq Y_T,\ \text{for some } Y\in\mathfrak{Y}_\mathbb{P}(y)\}$
\end{itemize}
Denote
\begin{equation}\label{cd}
\mathcal{C}_\mathbb{P}=\mathcal{C}_\mathbb{P}(1),\ \mathcal{D}_\mathbb{P}=\mathcal{D}_\mathbb{P}(1).
\end{equation}
It is easy to see that for $x>0,\ \mathcal{C}_\mathbb{P}(x)=x\mathcal{C}_\mathbb{P},\ \mathcal{D}_\mathbb{P}(x)=x\mathcal{D}_\mathbb{P}$. Define the value of the optimization problem under $\mathbb{P}\in\mathcal{P}$:
\begin{equation}
u_\mathbb{P}(x)=\sup_{c\in\mathcal{C}_\mathbb{P}(x)}\mathbb{E}_\mathbb{P}U(c),\ \ \ v_\mathbb{P}(y)=\inf_{d\in\mathcal{D}_\mathbb{P}(y)}\mathbb{E}_\mathbb{P}V(d).
\notag\end{equation}
Then define
\begin{equation}
u(x)=\inf_{\mathbb{P}\in\mathcal{P}}u_\mathbb{P}(x),\ \ \ v(y)=\inf_{\mathbb{P}\in\mathcal{P}}v_\mathbb{P}(y).
\notag\end{equation}
Below are the main results of this paper.
\begin{theorem}\label{theorem1}
Under \textbf{Assumption P}, we have
\begin{equation}\label{exchangeability}
u(x)=\hat u(x)=\inf_{\mathbb{P}\in\mathcal{P}}\sup_{(\Delta, h)}\mathbb{E}_\mathbb{P}U\left(x+(\Delta\cdot S)_T+h\cdot g\right),\ x>0.
\end{equation}
Besides, the value function $u$ and $v$ are conjugate, i.e., 
\begin{equation}
u(x)=\inf_{y>0}(v(y)+xy),\ \ \ v(y)=\sup_{x>0}(u(x)-xy).
\notag\end{equation}
\end{theorem}
\begin{theorem}\label{theorem2}
Let $x_0>0$. Under \textbf{Assumption P}, there exists a probability measure $\hat{\mathbb{P}}\in\mathcal{P}$, an optimal strategy $\hat X_T=x_0+(\hat\Delta\cdot S)_T+\hat h\cdot g\geq 0$, and $\hat Y_T\in\mathfrak{Y}_{\hat{\mathbb{P}}}(\hat y)$ with $\hat y=u'_{\hat{\mathbb{P}}}(x_0)$ such that 
\begin{itemize}
\item[(i)] $u(x_0)=u_{\hat{\mathbb{P}}}(x_0)=\mathbb{E}_{\hat{\mathbb{P}}}[U(\hat X_T)],$
\item[(ii)] $v(\hat y)=u(x_0)-\hat yx_0,$
\item[(iii)] $v(\hat y)=v_{\hat{\mathbb{P}}}(\hat y)=\mathbb{E}_{\hat{\mathbb{P}}}V(\hat Y_T)],$
\item[(iv)] $\hat X_T=I(\hat Y_T)$ and $\hat Y_T=U'(\hat X_T)$, $\hat{\mathbb{P}}$-a.s., and moreover $\mathbb{E}_{\hat{\mathbb{P}}}[\hat X_T\hat Y_T]=x_0\hat y$.
\end{itemize}
\notag\end{theorem}

\section{Proof of the main results}

This section is devoted to the proof of the main results, \thref{theorem1} and \thref{theorem2}.
\begin{proof}[Proof of \thref{theorem1}]
For $\mathbb{P}\in\mathcal{P}$ and any measurable function $f$ defined on $\mathbb{R}_+^j$, there exists a sequence of continuous functions $(f_n)_{n=1}^\infty$ converging to $f$ $\mathbb{P}$-a.s. (see e.g., Page 70 in \cite{Doob 1}). By a truncation argument, $f_n$ can be chosen to be bounded. Therefore, we have
\begin{equation}
\sup_{(\Delta, h)}\mathbb{E}_\mathbb{P}U\left(x+(\Delta\cdot S)_T+h\cdot g\right)=\sup_{(\Delta, h),\ \Delta\in C_b}\mathbb{E}_\mathbb{P}U\left(x+(\Delta\cdot S)_T+h\cdot g\right),
\notag\end{equation}
where $\Delta\in C_b$ means that each component $\Delta_j$ is a continuous bounded function on $\mathbb{R}_+^j,\ j=1,\dotso T-1$. Hence,
\begin{eqnarray}
\hat u(x)&=&\sup_{(\Delta, h)}\inf_{\mathbb{P}\in\mathcal{P}}\mathbb{E}_\mathbb{P}U\left(x+(\Delta\cdot S)_T+h\cdot g\right)\notag\\
&\geq&\sup_{(\Delta, h),\ \Delta\in C_b}\inf_{\mathbb{P}\in\mathcal{P}}\mathbb{E}_\mathbb{P}U\left(x+(\Delta\cdot S)_T+h\cdot g\right)\notag\\
\label{minimax}&=&\inf_{\mathbb{P}\in\mathcal{P}}\sup_{(\Delta, h),\ \Delta\in C_b}\mathbb{E}_\mathbb{P}U\left(x+(\Delta\cdot S)_T+h\cdot g\right)\\
&=&\inf_{\mathbb{P}\in\mathcal{P}}\sup_{(\Delta, h)}\mathbb{E}_\mathbb{P}U\left(x+(\Delta\cdot S)_T+h\cdot g\right)\notag\\
&\geq&\hat u(x),\notag
\end{eqnarray}
where \eqref{minimax} follows from the minimax theorem. Hence \eqref{exchangeability} is proved. The rest of this theorem can be proved following the arguments in the proofs of Lemmas 7 and 8 in \cite{Denis1}.
\end{proof}

\begin{theorem}\label{theorem 3.1}
Under \textbf{Assumption P}(2), for any given $\mathbb{P}\in\mathcal{P}$, the set $\mathcal{C}_\mathbb{P},\ \mathcal{D}_\mathbb{P}$ defined in \eqref{cd} satisfy the properties in Proposition 3.1 in \cite{Schachermayer2}.
\end{theorem}
\begin{proof}
It is obvious that (i) $\mathcal{C}_\mathbb{P}$ and $\mathcal{D}_\mathbb{P}$ are convex and solid, (ii) $\mathcal{C}_\mathbb{P}$ contains the constant function $1$, and (iii) for any  $c\in\mathcal{C}_\mathbb{P}, d\in\mathcal{D}_\mathbb{P}$, we have $\mathbb{E}_\mathbb{P}[cd]\leq 1$. We will finish the proof by showing the next four lemmas, where we use the notation $d\mathbb{Q}/d\mathbb{P}$ to denote both the Radon-Nikodym process and the Radon-Nikodym derivative on the whole space $\mathbb{R}_+^T$, whenever $\mathbb{Q}\sim\mathbb{P}$.
\begin{lemma}
$\mathcal{C}_\mathbb{P}$ is bounded in $L^0(P)$.
\end{lemma}
\begin{proof}
By \textbf{Assumption P}(2), there exists $\mathbb{Q}\in\mathcal{M}$ that is equivalent to $\mathbb{P}$. Then
\begin{equation}
\sup_{c\in\mathcal{C}_\mathbb{P}}\mathbb{E}_\mathbb{P}\left[\frac{d\mathbb{Q}}{d\mathbb{P}}c\right]=\sup_{c\in\mathcal{C}_\mathbb{P}}\mathbb{E}_\mathbb{Q}[c]\leq 1.
\notag\end{equation}
Hence,
\begin{eqnarray}
\sup_{c\in\mathcal{C}_\mathbb{P}}\mathbb{P}(c>K)&=&\sup_{c\in\mathcal{C}_\mathbb{P}}\mathbb{P}\left(\frac{d\mathbb{Q}}{d\mathbb{P}}c>\frac{d\mathbb{Q}}{d\mathbb{P}}K\right)\notag\\
&\leq&\sup_{c\in\mathcal{C}_\mathbb{P}}\left[\mathbb{P}\left(\frac{d\mathbb{Q}}{d\mathbb{P}}\leq\frac{1}{\sqrt{K}}\right)+ \mathbb{P}\left(\frac{d\mathbb{Q}}{d\mathbb{P}}c>\sqrt{K}\right)\right]\notag\\
&\leq&\mathbb{P}\left(\frac{d\mathbb{Q}}{d\mathbb{P}}\leq\frac{1}{\sqrt{K}}\right)+\frac{1}{\sqrt{K}}\sup_{c\in\mathcal{C}_\mathbb{P}}\mathbb{E}_\mathbb{P}\left[\frac{d\mathbb{Q}}{d\mathbb{P}}c\right]\notag\\
&\leq&\mathbb{P}\left(\frac{d\mathbb{Q}}{d\mathbb{P}}\leq\frac{1}{\sqrt{K}}\right)+\frac{1}{\sqrt{K}}\rightarrow 0,\ \ \ K\rightarrow\infty.\notag
\end{eqnarray}
\end{proof}
\begin{lemma}\label{lemma2}
For $c\in L_+^0(\mathbb{P})$, if $\mathbb{E}_\mathbb{P}[cd]\leq 1,\ \forall d\in\mathcal{D}_\mathbb{P}$, then $c\in\mathcal{C}_\mathbb{P}$.
\end{lemma}
\begin{proof}
It can be shown that for any $\mathbb{Q}\in\mathcal{M}(\mathbb{P}):=\{\mathbb{Q}\in\mathcal{M}:\ \mathbb{Q}\sim\mathbb{P}\}$, the process $\frac{d\mathbb{Q}}{d\mathbb{P}}$ is in $\mathfrak{Y}_\mathbb{P}$. Then for $c\in L_+^0(\mathbb{P})$
\begin{equation}
\sup_{\mathbb{Q}\in\mathcal{M}(\mathbb{P})}\mathbb{E}_\mathbb{Q}[c]=\sup_{\mathbb{Q}\in\mathcal{M}(\mathbb{P})}\mathbb{E}_\mathbb{P}\left[\frac{d\mathbb{Q}}{d\mathbb{P}}c\right]\leq 1.
\notag\end{equation}
Applying  the super-hedging Theorem on Page 6 in \cite{Nutz2}, we have that there exists a trading strategy $(\Delta, h)$, such that $1+(\Delta\cdot S)_T+h\cdot g\geq c,\ \mathbb{P}$-a.s., and thus $c\in\mathcal{C}_\mathbb{P}$.
\end{proof}
\begin{lemma}
For $d\in L_+^0(\mathbb{P})$, if $\mathbb{E}_\mathbb{P}[cd]\leq 1,\ \forall c\in\mathcal{C}_\mathbb{P}$, then $d\in\mathcal{D}_\mathbb{P}$.
\end{lemma}
\begin{proof}
Let $d\in L_+^0(\mathbb{P})$ satisfying $\mathbb{E}_\mathbb{P}[cd]\leq 1,\ \forall c\in\mathcal{C}_\mathbb{P}$. Then applying Proposition 3.1 in \cite{Schachermayer2} (here the space $\mathcal{C}_\mathbb{P}$ is lager than $\mathcal{C}$ defined in (3.1) in \cite{Schachermayer2}), we have that there exists $\tilde Y\in\mathcal{Y}_\mathbb{P}(1)$, such that $0\leq d\leq \tilde Y_T$. Define
$$
Y_k = \left\{ \begin{array}{rl}
 \tilde Y_k, &\mbox{ $k=0,\dotso,T-1,$} \\
  d,\  &\mbox{ $k=T.$}
       \end{array} \right.
$$
Then it's easy to show that $Y\in\mathfrak{Y}_\mathbb{P}$, and therefore $d\in\mathcal{D}_\mathbb{P}$ since $d=Y_T$.
\end{proof}
\begin{lemma}
$\mathcal{C}_\mathbb{P}$ and $\mathcal{D}_\mathbb{P}$ are closed in the topology of convergence in measure.
\end{lemma}
\begin{proof}
Let $\{c_n\}_{n=1}^\infty\subset\mathcal{C}_\mathbb{P}$ converge to some $c$ in probability with respect to $\mathbb{P}$. By passing to a subsequence, we may without loss of generality assume that $c_n\rightarrow c\geq 0,\ \mathbb{P}$-a.s. Then for any $d\in\mathcal{D}_\mathbb{P}$,
\begin{equation}
\mathbb{E}_\mathbb{P}[cd]\leq\liminf_{n\rightarrow\infty}\mathbb{E}_\mathbb{P}[c_nd]\leq 1,
\notag\end{equation}
by Fatou's lemma. Then from \leref{lemma2} we know that $c\in\mathcal{C}_\mathbb{P}$, which shows that $\mathcal{C}_\mathbb{P}$ is closed in the topology of convergence in measure. Similarily, we can show that $\mathcal{D}_\mathbb{P}$ is closed.
\end{proof}
\noindent The proof of \thref{theorem 3.1} is completed at this stage.
\end{proof}
\begin{proof}[Proof of \thref{theorem2}]
We use \thref{theorem 3.1} to show the second equalities in (i) and (iii), as well as (iv), by applying Theorems 3.1 and 3.2 in \cite{Schachermayer2}. The rest of the proof is purely convex analytic and can be done exactly the same way as that in the proofs of Lemmas 9-12 in \cite{Denis1}.
\end{proof} 

\section{A example of $\mathcal{P}$}
We will give an example of $\mathcal{P}$ satisfying \textbf{Assumption P} in this section. We assume that there exists $M>0$, such that
\begin{equation}
\mathcal{M}_M:=\{\mathbb{Q}\in\mathcal{M}:\ \mathbb{Q}(||S||_\infty>M)=0\}\neq\emptyset.
\notag\end{equation}
\begin{remark}
The assumption above is not restrictive. For example, if we are given a finite set of prices of European call options over a finite range of strike prices at each time period, and that the prices are consistent with an arbitrage-free model, then the model can be realized on a finite probability space, see \cite{Davis1} for details.
\end{remark}
Fix $\alpha\in(0,1)$ and $\beta\in(1,\infty)$. Let
\begin{equation}\label{egP}
\mathcal{P}:=\left\{\mathbb{P}:\ \mathbb{P}\sim\mathbb{Q}\text{ and }\alpha\leq\frac{d\mathbb{P}}{d\mathbb{Q}}\leq\beta,\ \text{for some }\mathbb{Q}\in\mathcal{M}_M\right\}.
\end{equation}
\begin{remark}
From the financial point of view, the boundedness condition on the Radon-Nikodym derivative $d\mathbb{P}/d\mathbb{Q}$ means that the physical measures, which represents the personal beliefs, should not be too far away from the martingale measures. 
\end{remark}
\begin{theorem}
$\mathcal{P}$ defined in \eqref{egP} satisfies \textbf{Assumption P}. 
\end{theorem}
\begin{proof}
It is obvious that $\mathcal{P}$ is convex, nonempty, and satisfies \textbf{Assumption P}(2). Let $(\mathbb{P})_{n=1}^\infty\subset\mathcal{P}$. Then there exists $\mathbb{Q}_n\in\mathcal{M}_M$ that is equivalent to $\mathbb{P}_n$ satisfying $\alpha\leq d\mathbb{P}_n/d\mathbb{Q}_n\leq\beta$. Since $(\mathbb{P}_n)$ and $(\mathbb{Q}_n)$ are supported on $[0,M]^T$, they are tight, and thus relatively weakly compact from Prokhorov's theorem. By passing to subsequences, we may without loss of generality assume that there exist probability measures $\mathbb{P}$ and $\mathbb{Q}$ supported on $[0,M]^T$, such that $\mathbb{P}_n\xrightarrow{\mbox{\tiny{w}}}\mathbb{P}$ and $\mathbb{Q}_n\xrightarrow{\mbox{\tiny{w}}}\mathbb{Q}$. Since probability measures have a compact support it can be shown using the monotone class theorem that $\mathbb{Q}\in\mathcal{M}_M$. Let $f$ be any nonnegative, bounded and continuous function. Then
\begin{equation}
\alpha\mathbb{E}_{\mathbb{Q}_n}[f]\leq\mathbb{E}_{\mathbb{P}_n}[f]\leq\beta\mathbb{E}_{\mathbb{Q}_n}[f].
\notag\end{equation} 
Letting $n\rightarrow\infty$, we have 
\begin{equation}
\alpha\mathbb{E}_\mathbb{Q}[f]\leq\mathbb{E}_\mathbb{P}[f]\leq\beta\mathbb{E}_\mathbb{Q}[f].
\notag\end{equation} 
Hence, $\mathbb{P}\in\mathcal{P}$, which completes the proof.
\end{proof}

\section{An extension}

Now instead of assuming that the market has a finite number of options, we can assume that our model is calibrated to a continuum of call options with payoffs $(S_i-K)^+,\ K\in\mathbb{R}_+$ at each time $t=i$, and price 
$$\mathcal{C}(i,K)=\mathbb{E}^\mathbb{Q}\left[(S_i-K)^+\right].$$
It is well-known that knowing the marginal $S_i$ is equivalent to knowing the prices $\mathcal{C}(i,K)$ for all $K\geq0$; see \cite{BL78}.  
Hence, we can assume that the marginals of the stock price $S=(S_i)_{i=1}^T$ are given by $S_i\sim\mu_i$, where $\mu_1,\dotso,\mu_T$ are probability measures on $\mathbb{R}_+$. Let 
$$\mathcal{M}:=\{\mathbb{Q} \text { probability measure on } \mathbb{R}_+^T:\ S=(S_i)_{i=1}^T \text{ is a } \mathbb{Q}-\text{martingale};$$
$$\text{ for } i=1,\dotso,T,\ S_i \text{ has marginal } \mu_i \text{ and mean } s_0\}.$$
If we make the similar assumptions as before on the utility function and $\mathcal{P}$, and if we can generalize the super-hedging theorem in \cite{Nutz2} to the case of infinitely many options (we only need the version with one probability measure here, which is much weaker result then the full generalization of \cite{Nutz2}), we can get similar results correspondingly. Indeed, this can be seen from the proof of \leref{lemma2}.

\bibliographystyle{siam}
\bibliography{ref}

\end{document}